\documentclass[aap,MSNbibl,citesort,dvips]{arximspdf}
\usepackage{stfloats}
\usepackage{graphicx}

\doi{10.1214/11-AAP786}
\volume{22}
\issue{4}
\pubyear{2012}
\firstpage{1301}
\lastpage{1327}

\makeatletter

\fnbelowfloat

\newtheorem{theorem}{Theorem}
\newtheorem{corollary}[theorem]{Corollary}

\newproclaim{definition}[theorem]{Definition}
\newproclaim{example}[theorem]{Example}

\newtheorem{lemma}[theorem]{Lemma}

\newtheorem{proposition}[theorem]{Proposition}

\newproclaim{remark}[theorem]{Remark}

\makeatother

\begin{document}
\begin{frontmatter}

\title{High order recombination and an application to cubature on Wiener space}
\runtitle{Stepping outside the Monte Carlo paradigm}

\begin{aug}
\author[A]{\fnms{C.} \snm{Litterer}\thanksref{t1}}
\and
\author[A]{\fnms{T.} \snm{Lyons}\corref{}\thanksref{t2}\ead[label=e1]{tlyons@maths.ox.ac.uk }}
\runauthor{C. Litterer and T. Lyons}
\affiliation{Imperial College London and University of Oxford}
\address[A]{Mathematical Institute\\
and \\
Oxford-Man Institute\\
\quad of Quantitative Finance\\
University of Oxford \\
24-29 St Giles'\\
Oxford, OX1 3LB\\
United Kingdom} 
\end{aug}

\thankstext{t1}{Supported by the Leverhulme trust Grant
F/08772/E. Final version of the paper prepared at
TU Berlin spported by ERC Grant 258237 (FP7/2007-2013).}

\thankstext{t2}{Supported by the EPSRC Grant EP/H000100/1.}

\received{\smonth{8} \syear{2010}}
\revised{\smonth{3} \syear{2011}}

%
\begin{abstract}
Particle methods are widely used because they can provide accurate
descriptions of evolving measures. Recently it has become clear that by
stepping outside the Monte Carlo paradigm these methods can be of
higher order with effective and transparent error bounds. A~weakness of
particle methods (particularly in the higher order case) is the
tendency for the number of particles to explode if the process is
iterated and accuracy preserved. In this paper we identify a new
approach that allows dynamic recombination in such methods and retains
the high order accuracy by simplifying the support of the intermediate
measures used in the iteration. We describe an algorithm that can be
used to simplify the support of a discrete measure and give an
application to the cubature on Wiener space method developed by Lyons
and Victoir [\textit{Proc. R. Soc. Lond. Ser. A Math. Phys. Eng. Sci.}
\textbf{460} (2004) 169--198].
\end{abstract}

%

\end{frontmatter}

\section{Introduction}

In pricing and hedging financial derivatives, as well as in assessing
the risk
inherent in complex systems, we often have to find approximations to
expectations of functionals of solutions to stochastic differential equations
(SDE). We consider a Stratonovich stochastic differential equation%
\[
d\xi_{t,x}=V_{0}(\xi_{t,x})\,dt+\sum_{i=1}^{d}V_{i}(\xi_{t,x})\circ dB_{t}
^{i},\qquad \xi_{0,x}=x,
\]
defined by a family of smooth vector fields $V_{i}$ and driven by Brownian
motion. It is well known that computing $P_{T-t}f:=E(f(\xi_{T-t,x}))$
corresponds to solving a parabolic partial differential equation (PDE). High
dimension and hypo-ellipticity are common obstacles that arise when one
calculates these quantities numerically. When facing these obstacles some
classical computational methods become unstable and/or intractable.

There are many settings where one is interested in tracking the
evolution of a measure over time in an effective numerical fashion. One
example is the numerical approximation to the solution of a linear parabolic
PDE. In this case, one tracks the evolution of the heat kernel measure
associated to the PDE. Another example is the filtering problem where one
wishes to approximate the unnormalized conditional distribution of the signal,
which is governed by a stochastic partial differential equation known
as the
Zakai equation.

An evolving measure can be viewed as a path in the space of measures. Thus,
even if the underlying state space is finite dimensional, we
potentially face
an infinite-dimensional problem. Particle approximations can, in many cases,
provide good descriptions of evolving measures (see, e.g., the survey
articles \cite{Cr,CrDou1}). Higher order methods may allow us to take far
fewer time steps than classical methods in the approximations. An
example of a~higher order particle method may be found in Kusuoka \cite{kusuoka3}. Although
effective in practice (compare Ninomiya \cite{ninomiya} and Ninomiya
and Victoir
\cite{ninomiya-victoir}), these methods have the drawback that the
number of
particles can explode exponentially if the process is iterated and accuracy
preserved (see, e.g., Lyons and Victoir \cite{lyons}).

Sometimes the essential properties of a probability measure we care
about can
accurately be described and captured by the expectations of a finite
set of
test functions. If we can find such a family of test functions we can replace
the original measure with a simpler measure with smaller support that
integrates all test functions correctly and hence, still has the right
properties, provided, of course, the number of test functions is small compared
to the cardinality of the support of the original measure. We will also insist
that the reduced measure $\tilde{\mu}$ has $\operatorname{supp}(\tilde{\mu})\subseteq
\operatorname{supp}(\mu)$. This condition ensures that feasibility constraints imposed
on the
measure $\mu$ will also be satisfied by $\tilde{\mu}$. For a finite Borel
measure $\mu$ on a polish space $\Omega$ and a set of integrable functions
$\{p_{1},\ldots,p_{n}\}$, we can show that such a reduced measure
$\tilde{\mu}$
always exists with $\operatorname{card}(\operatorname{supp}(\tilde{\mu}))\leq n+1$.

In this
paper we present a simple algorithm that can be used to compute reduced
measures for discrete measures $\mu$. The runtime is polynomial in the
size of
the support of the measure $\mu$. The algorithm relies on the
observation that
if $P$ is the $R^{n}$ valued random variable $P(x):=(p_{1}(x),\ldots
,p_{n}(x))$ and~$\mu_{P}$ the law of $P$ under the measure $\mu$, then finding
a reduced measure $\tilde{\mu}$ is equivalent to finding $\tilde{\mu
}_{P}$ a
discrete measure on $R^{n}$ with $\operatorname{card}(\operatorname{supp}
(\tilde{\mu}_{P})) =n+1$
and the
same center of mass (CoM) as $\mu_{P}$.

We describe an application to the Kusuoka--Lyons--Victoir (or KLV
cubature on Wiener space) method developed by Lyons and Victoir
\cite{lyons}, following Kusuoka~\cite{kusuoka3}. It provides higher
order approximations to $E(f(\xi_{T,x}))$ if the test function $f$ is
Lipschitz and the vector fields satisfy Kusuoka's UFG condition (see
\cite{kusuoka2}) which is weaker than the usual H\"{o}rmander
condition. The expectation $E(f(\xi_{T,x}))$ might be viewed as an
infinite-dimensional integral against Wiener measure. The authors
construct discrete cubature measures
$Q_{T}=\sum_{i=1}^{n}\lambda_{i}\delta _{\omega_{i,T}}$ supported on
continuous paths of bounded variation that approximate Wiener measure
in the sense that they integrate iterated integrals up to a fixed
degree correctly. The expectation of a Wiener functional~$f(\xi_{T,x})$
against the discrete cubature measure may be obtained by computing the
endpoints of the solution of the SDE along the paths in the support of
$Q_{T}$. Thus the KLV method might be viewed as a discrete Markov
kernel taking discrete measures on $R^{N}$ to discrete measures on
$R^{N}$. More explicitly we have
\[
\operatorname{KLV}(\delta_{x},T)=\sum_{i=1}^{n}\lambda_{i}\delta_{\xi_{T,x}(\omega_{i,T})}
\]
and
\[
E_{Q_{T}}f(\xi_{T,x})=E_{\mathrm{KLV}(\delta_{x},T)}f.
\]
The bound on the error when replacing the Wiener measure with a cubature
measure is given in terms of higher order derivatives of $f$, so in general
will not be small as $f$ is only assumed to be Lipschitz. The results in
Kusuoka and Stroock \cite{kusuoka1} and Kusuoka \cite{kusuoka2} show that
$P_{t}f$ will be smooth, at least in the direction of the vector fields $V_{i}
$. This is resolved by applying the method iteratively over a partition
of the
time interval $[0,T]$. The operator corresponding to the iterated application
of the KLV method is Markov and hence, the error of the approximation of
$P_{T}f$ on the global time interval $[0,T]$ is the sum of the error of the
approximations over the subintervals of the partition. So considering an
uneven partition of the global time interval $[0,T]$ with time steps getting
smaller toward the end, we can iteratively apply the cubature method
over the
subintervals and reduce the error in the approximation to any accuracy.
If $m$
is the degree of the cubature formula, we can find a partition such
that the
error in the weak approximation is uniformly bounded by
\[
Ck^{-(m-1)/2}\Vert\nabla f\Vert_{\infty},
\]
where $k$ is the number of time steps in the partition and $C$ a constant
independent of $k$ and $f$.

The iterated KLV method might be
viewed as a particle system on $R^{N}$ where the particles branch in an
$n$-ary tree. Hence, the number of ODEs to solve grows exponentially in the
number of iterations. In this paper we add recombination to the KLV method.
After each application of the KLV operation we replace the intermediate
measures by reduced measures. The property of the KLV measure we are targeting
is to integrate $P_{t}f$, the heat kernel applied to $f$, correctly. We have
identified a finite set of test functions that ensures that the bound
on the
overall error of the approximation of $P_{T}f$ is only increased by a constant
factor and hence, the modified method has the same convergence properties.
Moreover, we can show that under the H\"{o}rmander condition for
bounded vector
fields the number of test functions required grows polynomially in the number
of iterations.

We finish the paper with a toy numerical example that illustrates how one
blends the methods of this paper together in a concrete example to
compute a
solution of a one-dimensional PDE to high accuracy when the boundary
data is
piecewise smooth and the discontinuities are not known to the PDE solver.

We believe that the combination of the two ideas---higher order
particle methods to describe the evolution of a measure on the one hand and
simplifying the support of the measures used in the description, by
characterizing essential properties of a measure using the expectations
of a
finite set of test functions on the other hand---have more general
applications than investigated so far. Applications to the stochastic
filtering problem appear to be particularly promising (see Litterer and Lyons
\cite{litterer,litterer2} for an outline).

\section{A reduction algorithm for the support of a discrete
measure}\label{sectionalgo1}

Let us start the precise description of the reduction problem. The
notation in
this section is independent of the notation used in the description of the
cubature method in the following sections. Consider a finite set of test
functions $P_{n}=\{p_{1},\ldots,p_{n}\}$ on $(\Omega,\mu)$, a measure space
with $\mu$ a finite discrete measure
\[
\mu=\sum_{i=1}^{\hat{n}}\lambda_{i}\delta_{z_{i}},\qquad \lambda
_{i}>0,z_{i}%
\in\Omega,
\]
with large support. By this we mean that $\hat{n}$ is at least of order
$n^{2}$. In the following we assume that $\mu$ is a probability
measure, that is,
the weights add up to one.
%
\begin{definition}
\label{def-reduced-measure}We will call a discrete probability measure
$\tilde{\mu}$ a reduced measure with respect to $\mu$ and $P_{n}$ if it
satisfies the following three conditions:

\begin{longlist}[(3)]
\item[(1)]
$\operatorname{supp}(\tilde{\mu
})\subseteq \operatorname{supp}(\mu)$.

\item[(2)]
For all $p\in P_{n}$
\[
\int p(x)\tilde{\mu}(dx)=\int p(x)\mu(dx).
\]

\item[(3)]
$\operatorname{card}(\operatorname{supp}(\tilde{\mu}))\leq
n+1$.\vadjust{\goodbreak}
\end{longlist}
\end{definition}

The first condition is more important than it looks as it ensures that
feasibility constraints imposed on samples drawn from $\mu$ will also be
satisfied by $\tilde{\mu}$. We wish to construct effective algorithms to
compute the reduced measure.

Let $P$ be the $R^{n}$-valued random variable $P:=(p_{1},\ldots,p_{n})$
defined on $(\Omega,\mu)$. Then the law $\mu_{P}$ of $P$ is the discrete
measure on $R^{n}$
%
\begin{equation} \label{definitionmeasure}%
\mu_{P}=\sum_{i=1}^{\hat{n}}\lambda_{i}\delta_{x_{i}},\qquad  x_{i}=(p_{1}
(z_{i}),\ldots,p_{n}(z_{i}))^{T}\in R^{n}.
\end{equation}
The center of mass (CoM) for the measure $\mu_{P}$ is given by
%
\begin{equation} \label{eqmass}%
\operatorname{CoM}(\mu_{P})=\sum_{i=1}^{\hat{n}}\lambda_{i}x_{i}.
\end{equation}
To find a reduced measure we articulate an equivalent problem in terms
of~$\mu_{P}$. The problem becomes finding a subset $x_{i_{k}}$ of the points
$x_{i}$ and positive weights $\tilde{\lambda}_{i_{k}}$ to produce a new
probability measure $\tilde{\mu}_{P}$ such that $\operatorname{CoM}(\tilde{\mu}_{P}%
)=\operatorname{CoM}(\mu_{P})$. A~reduced measure $\tilde{\mu}$ is then easily
obtained from~$\tilde{\mu}_{P}$ by taking\looseness=1
\[
\tilde{\mu}=\sum\tilde{\lambda}_{i_{k}}\delta_{z_{i_{k}}}
\]\looseness=0
with $z_{i_{k}}\in \operatorname{supp}(\mu)$ satisfying $P(z_{i_{k}})=x_{i_{k}}$.

Note that given any subset $x_{i_{k}}$ there exist suitable
weights $\tilde{\lambda}_{i_{k}}$ if and only if $\operatorname{CoM}(\mu_{P})$ is contained
in the convex hull of these points. Caratheodory's theorem implies that in
principle one can always find $\tilde{\mu}_{P}$ with support having
cardinality at most $n+1$ and the algorithm explained below provides
a~constructive proof to that.

By considering $x_{i}-\operatorname{CoM}(\mu
_{P})$ in place of the $x_{i}$, we may assume without loss of
generality that
$\operatorname{CoM}(\mu_{P})$ is at the origin. We may also assume that the $x_{i}$
are all
distinct, as we can otherwise eliminate points $x_{i}$ from the original
measure $\mu$ by sorting and combining them.

A first
algorithm (Algorithm 1), communicated to us by Victoir \cite{victoir},
sequentially eliminates particles from the support of the measure. It
is well
known and has, for example, been used in constructive proofs of Tchakaloff's
theorem (Davis~\cite{davis}).

Given any $n+2$ points, the system
given by
%
\begin{eqnarray} \label{eqrel}%
\sum_{i=1}^{n+2}u_{i}x_{k_{i}}&=&0,
\\
\sum_{i=1}^{n+2}u_{i}&=&0\nonumber
\end{eqnarray}
is a linear system with $n+2$ variables, but only $n+1$ constraints. Therefore,
it has a nontrivial solution, which may, for example, be determined
using
Gaussian elimination. Thus we may either add
\[
\min_{u_{i}<0}\biggl( -\frac{\lambda_{i}}{u_{i}}\biggr) \sum_{j=1}^{n+2}%
u_{j}x_{k_{j}}%
\]
to (\ref{eqmass}) or subtract
\[
\min_{u_{i}>0}\biggl( \frac{\lambda_{i}}{u_{i}}\biggr) \sum_{j=1}^{n+2}%
u_{j}x_{k_{j}}
\]
from (\ref{eqmass}) leaving all weights in the result nonnegative and
their overall sum unchanged. In either case, by construction, the coefficient
of some $x_{j}$ vanishes. We now have obtained a new probability
measure with
the same center of mass and at least one point less in the support. Applying
the procedure iteratively until there are only $n+1$ points left, we
obtain a
reduced measure. Clearly the method requires no more than $\hat{n}$ iterations
of the above procedure.
%
\begin{remark}
\label{rem1} If $\tilde{n}$ is the dimension of the lowest-dimensional
(affine) subspace of $R^{n}$ containing the set $\{(p_{1}(y),\ldots
,p_{n}(y)) | y\in \operatorname{supp}(\mu)\}$, we can continue to apply the elimination
procedure described in Algorithm $1$ until $\operatorname{card}(\operatorname{supp}(\tilde{\mu}))\leq
\tilde
{n}+1$.
\end{remark}

For improving the order of the overall algorithm we now look at suitable
linear combinations instead of points.

To describe the
algorithm we define an abstract procedure $A$ that takes a discrete
probability measure $\nu$ with $2(n+1)$ particles in its support and returns
another discrete probability measure $\tilde{\nu}$ with $(n+1)$
particles in
its support satisfying $\operatorname{CoM}(\nu)=\operatorname{CoM}(\tilde{\nu})$ and $\operatorname{supp}(\tilde{\nu
})\subseteq \operatorname{supp}(\nu)$. Procedure $A$ may, for example, be realized by $n+1$
applications of the reduction procedure of Algorithm 1.\vspace*{8pt}

\textit{Main reduction algorithm} (\textit{Algorithm} 2):
(1) Partition the support
of $\mu_{P}=\sum_{i=1}^{\hat{n}}\lambda_{i}\delta_{x_{i}}$ into
$2(n+1)$ sets
of as near equal size as possible. Let these sets be denoted by $I_{j}$,
$1\leq j\leq2(n+1)$.

(2) Compute the probability measure $\nu
=\sum_{i=1}^{2(n+1)}\nu_{i}\delta_{\tilde{x}_{i}}$ where
\[
\tilde{x}_{j}=E_{\mu_{P}}( x | x\in I_{j}) =\sum_{x_{i}\in
I_{j}}\frac{\lambda_{i}x_{i}}{\nu_{_{j}}}
\]
and $\nu_{j}=\mu_{P}(I_{j})=\sum_{i \dvtx x_{i}\in I_{j}}\lambda_{i}$.

(3) Apply procedure $A$ to compute a measure $\tilde{\nu}=\sum
_{j=1}^{n+1}\tilde{\nu}_{i_{j}}\delta_{\tilde{x}_{i_{j}}}$ with $\operatorname{CoM}(\nu
)=\operatorname{CoM}(\tilde{\nu})$.

(4) Repeat (1)--(3) with
\[
\mu_{P}^{\prime}=\sum_{j=1}^{n+1}\sum_{x_{k}\in I_{i_{j}}}\tilde{\nu
}_{i_{j}%
}\frac{\lambda_{k}}{\nu_{i_{j}}}\delta_{x_{k}}
\]
for $\mu_{P}$ until $n+1$ particles are left in the support of
$\mu_{P}$.
%
\begin{proposition}
\label{algo1} Given $\mu$ and $P_{n}$, the algorithm described above requires
$\lceil\lg(\hat{n}/n)\rceil$ iterations of procedure $A$ to compute a
reduced measure.
\end{proposition}
\begin{pf}
We might interpret the points $\tilde{x}_{j}$ as the respective center of
masses of the individual subsets $I_{j}$.

It is clear that $\mu _{P}^{\prime}$ has positive weights and support
contained in the support of~$\mu_{P}$. Hence, we only need to show that
$\operatorname{CoM}(\mu_{P}^{\prime })=\operatorname{CoM}(\mu _{P})$.

We have
\begin{eqnarray*}
\operatorname{CoM}(\mu_{P}^{\prime})&=&\sum_{j=1}^{n+1}\tilde{\nu}_{i_{j}}\sum_{x_{k}\in
I_{i_{j}}}\frac{\lambda_{k}x_{k}}{\nu_{j}}=\sum_{j=1}^{n+1}\tilde{\nu
}_{i_{j}%
}\tilde{x}_{i_{j}}%
\\
&=&\operatorname{CoM}(\tilde{\nu})=\operatorname{CoM}(\nu)=\sum_{j=1}^{2(n+1)}\nu_{j}\tilde{x}_{j}=\sum
_{j=1}^{2(n+1)}\nu_{j}\sum_{x_{i}\in I_{j}}\frac{\lambda_{i}x_{i}}{\nu
_{j}}\\
&=&\operatorname{CoM}(\mu_{P}).
\end{eqnarray*}
As $\hat{n}\leq n2^{\lceil\lg(\hat{n}/n)\rceil}$, we may assume without
loss of
generality that $\hat{n}=n2^{\lceil\lg(\hat{n}/n)\rceil}$. It is
obvious that
each iteration halves the number of particles in the support of $\mu
_{P}$ and
we require exactly $\lceil\lg(\hat{n}/n)\rceil$ iterations.
\end{pf}
%
\begin{corollary}
Using the main reduction algorithm we can compute a~reduced measure with
respect to $\mu$ and $P_{n}$ in
\[
O\bigl(n\hat{n}+n\log(\hat{n}/n)C(n+2,n+1)\bigr)
\]
steps where $C(n+2,n+1)$ represents the number of steps required to
solve a
system of linear equations with $n+2$ variables and $n+1$ constraints.
\end{corollary}
\begin{pf}
To compute the intermediate measures $\nu$, we need to calculate
$n$-dimensional linear combinations. The number of steps required for these
additions is bounded above by the series
\[
n\sum_{i=0}^{\infty}\hat{n}2^{-i}=2n\hat{n}.
\]
The procedure $A$ may be realized by $n+1$ applications of the reduction
procedure used in Algorithm 1 described above.
\end{pf}
%
\begin{remark}
Note that the linear systems of equations we need to solve in the algorithm
are singular. Hence, for a practical implementation we have used a method
based on the singular value decomposition (SVD) to avoid numerical
instability.\setcounter{footnote}{2}\footnote{A dll with an implementation of a version of the
algorithm and a Visual Studio project with a simple example for its use can
currently be found at
\texttt{\href{http://www.maths.ox.ac.uk/\textasciitilde tlyons/Recombination/reduce\_dist\_01\_paper.zip}{http://www.maths.ox.ac.uk/}
\href{http://www.maths.ox.ac.uk/\textasciitilde tlyons/Recombination/reduce\_dist\_01\_paper.zip}{\textasciitilde tlyons/Recombination/reduce\_dist\_01\_paper.zip}}.}
\end{remark}

If the support of the measure $\mu$ we wish to target is particularly
large or
possibly even infinite, we can consider a different approach. If we can
find a~subset of points that with a reasonably high probability contains the
CoM in
its convex hull, we may use linear programming to check if a given set of
points contains the CoM in its convex hull and reconstruct the weights. The
results in Wendel \cite{wendel} imply, for example, that a collection
of $k$
uniform i.i.d. random variables on the unit sphere in $R^{N}$ contains
the origin
with probability
\[
P_{N,k}=1-2^{-k+1}\sum_{j=0}^{N-1}\pmatrix{k-1\cr j}.
\]
In particular this yields $P_{N,2N}=1/2$.

\section{Outline of the cubature algorithm}
\label{SectionCubAlg}

We describe the cubature method developed by Lyons and Victoir
\cite{lyons}. Throughout the paper, $C$ is a constant that may change
from line to line; specific constants, however, will be indexed $C_{1},
C_{2},\ldots.$ Let $C_{b}^{\infty}(R^{N},R^{N})$ denote the smooth
bounded $R^{N}$ valued functions whose derivatives of any order are
bounded. Then $V_{i}\in C_{b}^{\infty}(R^{N},R^{N}), 0\leq i\leq d$, may
be regarded as vector fields on $R^{N}$. We define a partial
differential operator $L=V_{0}+\frac{1}{2}(V_{1}^{2}+\cdots
+V_{d}^{2})$ and consider the following parabolic partial differential
equation (PDE)
%
\begin{eqnarray}\label{pde}%
\frac{\partial u}{\partial t}(t,x)&=&-Lu(t,x),\nonumber\\[-8pt]\\[-8pt]
u(T,x)&=&f(x)
\nonumber
\end{eqnarray}
for a given Lipschitz function $f$. The aim is to find an approximation of
$u(0,x)$ for a given $x$. Consider the probability space $(C_{0}%
^{0}([0,T],R^{d}),\mathcal{F},\mathbb{P})$, where
$C_{0}^{0}([0,T],R^{d})$ is
the space of $R^{d}$ valued continuous functions starting at $0$,
$\mathcal{F}$ its usual Borel $\sigma$-field and $\mathbb{P}$ the Wiener
measure. Define the coordinate mapping process $B_{t}^{i}(\omega)=\omega
^{i}(t)$ for $t\in\lbrack0,T]$, $\omega\in\Omega$. Under Wiener measure,
$B=(B_{t}^{1},\ldots,B_{t}^{d})$ is a Brownian motion starting at zero.
Furthermore, let $B_{t}^{0}(t)=t$. Let $\xi_{t,x}$, $t\in\lbrack0,T]$,
$x\in
R^{N}$ be a version of the solution of the Stratonovich stochastic
differential equation (SDE)
%
\begin{equation}\label{sde}%
d\xi_{t,x}=\sum_{i=0}^{d}V_{i}(\xi_{t,x})\circ dB_{t}^{i},
\qquad\xi_{0,x}=x,
\end{equation}
that coincides with the pathwise solution on continuous paths of bounded
variation. In this case, classical theory tells us that $u(t,x)=E(f(\xi
_{T-t,x}))$ is the solution to~(\ref{pde}).

We define the It\^o
functional $\Phi_{T,x}\dvtx C_{0}^{0}([0,T],R^{d})\rightarrow R^{N}$ by
%
\begin{equation} \label{ito-functional}%
\Phi_{T,x}(\omega)=\xi_{T,x}(\omega).
\end{equation}
Denote by $R_{m}[X_{1},\ldots,X_{d}]$ the space of polynomials\footnote{Any
finite-dimensional space of integrable and continuous functions could
be used
to define cubature.
This extension can be helpful.} in $d$ variables having degree less or
equal to
$m$. Let $\mu$ be a positive Borel measure on $R^{d}$. A discrete measure
$\tilde{\mu}$
\[
\tilde{\mu}=\sum_{i=1}^{n}\lambda_{i}\delta_{x_{i}}%
\]
with $x_{1},\ldots,x_{n}$ contained in $\operatorname{supp}(\mu)$ satisfies a cubature
formula of degree $m$ if and only if for all polynomials $P\in R_{m}%
[X_{1},\ldots,X_{d}]$,
\[
\int_{R^{d}}P(x)\mu(dx)=\int_{R^{d}}P(x)\tilde{\mu}(dx)=\sum_{i=1}^{n}%
\lambda_{i}P(x_{i}).
\]

It is well known that if all moments of $\mu$ up to degree $m$ exist we
can always find such a measure with
\[
\operatorname{card}(\operatorname{supp}(\mu))\leq\operatorname{dim}(R_{m}[X_{1},\ldots,X_{d}])+1
\]
(see, e.g., Bayer and Teichmann \cite{teichmann}). More generally we
have the
following lemma, which we state without proof.
%
\begin{lemma}
Let $\Omega$ be a polish space, $\mathcal{F}$ its Borel sets and $\mu$
a Borel
probability measure on $(\Omega,\mathcal{F})$. Let $f_{1},\ldots,f_{n}$
be a
finite sequence of real-valued Borel measurable functions on the probability
space $(\Omega,\mathcal{F},\mu)$ with $E(|f_{i}|)<\infty$ for $1\leq
i\leq n$.
Moreover, suppose that $D$ is a Borel set with $\mu(D)=1 $. Then there exist
points $w_{1},\ldots,w_{n+1}\in D$ and a discrete measure
\[
\tilde{\mu}=\sum_{i=1}^{n+1}\lambda_{i}\delta_{w_{i}}%
\]
such that
\[
E_{\mu}(f_{i})=E_{\tilde{\mu}}(f_{i})
\]
for $1\leq i\leq n$.
\end{lemma}

In other words, $\mu$ admits a reduced measure $\tilde{\mu}_{P}$ with
respect to any finite set $P$ of integrable functions. In connection
with the
use of the Taylor formula, a~cubature measure provides an effective
tool for
integration over finite-dimensional spaces.

One can formulate an analogous condition to identify cubature measures on
Wiener space. Here the role of polynomials is taken by iterated
integrals of
the form
\[
\int_{0<t_{1}<\cdots<t_{k}<T}\circ dB_{t_{1}}^{i_{1}}\cdots{\circ dB_{t_{k}
}^{i_{k}}}.
\]
We identify this iterated integral by the multi-index
$(i_{1},\ldots,i_{k})$.

Define the set of all multi-indices $A$ by
\[
A=\bigcup_{k=0}^{\infty}\{0,\ldots,d\}^{k}%
\]
and let $\alpha=(\alpha_{1},\ldots,\alpha_{k})\in A$ be a multi-index.
Furthermore, we define a~degree on a multi-index $\alpha$ by $\Vert
\alpha
\Vert=k+\operatorname{card}(j\dvtx\alpha_{j}=0)$ and let
\[
A(j)=\{\alpha\in A\dvtx\Vert\alpha\Vert\leq j\}.
\]

Moreover, define $A_{1}$ by $A_{1}=A\setminus\{\varnothing,(0)\}$ and let
\mbox{$A_{1}(j)=\{\alpha\in A_{1}\dvtx\Vert\alpha\Vert\leq j\}$}. It follows from the
scaling property of Brownian motion that
\[
\int_{0<t_{1}<\cdots<t_{k}<T}\circ dB_{t_{1}}^{\alpha_{1}}\cdots{\circ
dB_{t_{k}%
}^{\alpha_{k}}}
\]
equals, in law,
%
\begin{equation} \label{scaling}%
T^{\Vert\alpha\Vert/2}\int_{0<t_{1}<\cdots<t_{k}<1}\circ
dB_{t_{1}}^{\alpha
_{1}}\cdots{\circ dB_{t_{k}}^{\alpha_{k}}}.
\end{equation}
%
\begin{definition}
Fix a finite set of multi-indices $\tilde{A}\subseteq A$. We say that
a~discrete measure $Q_{T}$ assigning positive weights $\lambda_{1}%
,\ldots,\lambda_{n}$ to paths
\[
\omega_{1},\ldots,\omega_{n}\in C_{0,bv}^{0}([0,T],R^{d})
\]
is a cubature measure if for all $(i_{1},\ldots,i_{k})\in\tilde{A}$,
\[
E\biggl( \int_{0<t_{1}<\cdots<t_{k}<T}\circ dB_{t_{1}}^{i_{1}}\cdots{\circ
dB_{t_{k}}^{i_{k}}}\biggr) =\sum_{j=1}^{n}\lambda_{j}\int_{0<t_{1}<\cdots
<t_{k}<T}d\omega_{j}^{i_{1}}(t_{1})\cdots d\omega_{j}^{i_{k}}(t_{k}),
\]
where the expectation is taken under Wiener measure. If $\tilde
{A}=A(m)$ we
say that
\[
Q_{T}=\sum_{j=1}^{n}\lambda_{j}\delta_{\omega_{j}}%
\]
is cubature measure of degree $m$.
\end{definition}

In \cite{lyons}, the authors show that one can always find a cubature measure
supported on, at most, $\operatorname{card}(\tilde{A})$ continuous paths of bounded variation.
More importantly, they give an explicit construction of a degree $5 $ cubature
formula with $O(d^{3})$ paths in its support.

Suppose paths
$\omega_{1},\ldots,\omega_{n}$ and weights $\lambda_{i}$ define a cubature
measure for $T=1$. It follows immediately from (\ref{scaling}) that the
measure supported on paths~$\omega_{T,i}$ given by
%
\begin{equation}\label{rescaled-paths}%
\omega_{T,i}^{j}=\sqrt{T}\omega_{i}^{j}(t/T),\qquad j=1,\ldots,d,
\end{equation}
and unchanged weights $\lambda_{i}$ defines a cubature measure for general
$T$. From now on suppose that the measure $Q:=Q_{1}$ is a cubature
measure of degree~$m$.

The following proposition, taken from
\cite{lyons}, is the key step in estimating the error $E_{T}$ when one
approximates the expectation of $f(\xi_{T,x})$ under the Wiener measure
by its
expectation against $Q$.
%
\begin{proposition}
\label{eq1}
\begin{eqnarray*}
E_{T}:\!&=&\sup_{x\in R^{n}}\Biggl\vert Ef(\xi_{T,x})-\sum_{i=1}^{n}\lambda
_{i}f(\Phi_{T,x}(\omega_{T,i}))\Biggr\vert
\\
&\leq& {C\sum_{j=m+1}^{m+2}T^{j/2}\sup_{(\alpha_{1},\ldots,\alpha_{i})\in
A(j)\setminus A(j-1)}}\Vert V_{\alpha_{1}}\cdots V_{\alpha_{i}}f\Vert
_{\infty},
\end{eqnarray*}
where $C$ is a constant that only depends on $d$, $m$ and $Q_{1}$.
\end{proposition}

In general, the right-hand side of the inequality in Theorem \ref{eq1}
is not
sufficient to directly obtain a good error bound for the approximation
of the
expectation, in particular if $f$ is only assumed to be Lipschitz, the estimate
appears useless. So, instead of approximating
\[
P_{T}f(x):=E(f(\xi_{T,x}))
\]
in one step, one considers a partition $\mathcal{D}$ of the interval $[0,T]$
\[
t_{0}=0<t_{1}<\cdots<t_{k}=T
\]
with $s_{j}=t_{j}-t_{j-1}$ and solves the problem over each of the smaller
subintervals by applying the cubature method recursively. If $\tau$ and
$\tau^{\prime}$ are two path segments,\vadjust{\goodbreak} we denote their concatenation by
$\tau\otimes\tau^{\prime}$. For the approximation, we consider all possible
concatenations of cubature paths over the subintervals, that is, all
paths of the
form $\omega_{s_{1},i_{1}}\otimes\cdots\otimes\omega_{s_{k},i_{k}}$. We define
a~corresponding probability measure $\nu$ by
\[
\nu=\sum_{i_{1},\ldots,i_{k}=1}^{n}\lambda_{i_{1}}\cdots\lambda_{i_{k}}%
\delta_{\omega_{s_{1},i_{1}}\otimes\cdots\otimes\omega_{s_{k},i_{k}}}.
\]
The following theorem taken from Lyons and Victoir \cite{lyons} is the
main error
estimate for the iterated cubature method, which we in the following also
refer to as the Kusuoka--Lyons--Victoir (KLV) method.
%
\begin{theorem}
\label{klv-error-thm}The total error $E_{\mathcal{D}}$ for the approximation
\begin{eqnarray*}
E_{\mathcal{D}}:\!&=&{\sup_{x\in R^{N}}}\vert P_{T}f-E_{\nu}(f(\xi
_{T,x}))\vert
\\
&=&\sup_{x\in R^{N}}\Biggl\vert P_{T}f(x)-\sum_{i_{1}=1}^{n}\cdots\sum
_{i_{k}%
=1}^{n}\lambda_{i_{1}}\cdots\lambda_{i_{k}}f\bigl( \Phi_{T,x}(\omega
_{s_{1},i_{1}}\otimes\cdots\otimes\omega_{s_{k},i_{k}})\bigr)
\Biggr\vert
\end{eqnarray*}
is bounded by
%
\begin{equation} \label{iterCub}%
C_{1}( T) \Vert\nabla f\Vert_{\infty}\Biggl( s_{k}^{1/2}%
+\sum_{j=m}^{m+1}\sum_{i=1}^{k-1}\frac{s_{i}^{(j+1)/2}}{(T-t_{i})^{j/2}%
}\Biggr) ,
\end{equation}
where $C_{1}( T) $ is a constant independent of $f$ and $k$, the
number of time steps in the partition of the time interval $[0,T]$.
\end{theorem}

To compute the expectation with respect to the measure $\nu$ exactly requires
one to solve
\[
\frac{n^{k+1}-1}{n-1}
\]
inhomogeneous ODEs (each corresponding to a path $\omega_{s_{1},i_{1}}%
\otimes\cdots\otimes\omega_{s_{k},i_{k}}$) where $n$ denotes the number of
paths in the support of the cubature measure~$Q$ and $k $ the number of
subintervals in the partition. Hence, the number of ODEs to solve grows
exponentially in the number of iterations.

Following Kusuoka \cite{kusuoka2}, we define for multi-indices $\alpha
=(\alpha_{1},\ldots,\alpha_{k}),\beta=(\beta_{1},\ldots,\beta_{l})\in
A$ a
multiplication by
\[
\alpha\ast\beta=(\alpha_{1},\ldots,\alpha_{k},\beta_{1},\ldots,\beta_{l}).
\]
We inductively define a family of vector fields indexed by $A$ by taking
\begin{eqnarray*}
V_{[\varnothing]}&=&0,\qquad  V_{[i]}=V_{i},\qquad 0\leq i\leq d,
\\
V_{[\alpha\ast i]}&=&\bigl[V_{[\alpha]},V_{i}\bigr],\qquad 0\leq i\leq d,\alpha\in A.
\end{eqnarray*}
The main ingredients used when obtaining the bound (\ref{iterCub}) are
Proposition~\ref{eq1} and the following\vadjust{\goodbreak} regularity result due to
Kusuoka and
Stroock \cite{kusuoka1} and Kusuoka \cite{kusuoka2}, which says that
even if
$f$ is not smooth, $P_{s}f$ is smooth in the directions of the vector fields
$V_{i}$. Let $f$ be Lipschitz and $\alpha_{1},\ldots,\alpha_{k}\in
A_{1}$, then
for all $t\in(0,1]$,
%
\begin{equation}
\bigl\Vert V_{[\alpha_{1}]}\cdots V_{[\alpha_{k}]}P_{t}f\bigr\Vert_{\infty}\leq
\frac{Ct^{1/2}}{t^{(\Vert\alpha_{1}\Vert+\cdots+\Vert\alpha_{k}\Vert
)/2}}%
\Vert\nabla f\Vert_{\infty}
\end{equation}
provided the vector fields satisfy the UFG condition defined below.

Following Kusuoka \cite{kusuoka2} we introduce a condition on
the vector fields.
%
\begin{definition}
The family of vector fields $V_{i}$, $i=0,\ldots,d$, is said to satisfy the
condition (UFG) if the Lie algebra generated by it is finitely
generated as a
$C_{b}^{\infty}$ left module, that is, there exists a positive $k$ and
$u_{\alpha,\beta}\in C_{b}^{\infty}$ satisfying for all $\alpha\in A_{1}$,
%
\begin{equation}\label{UFGEq}%
V_{[\alpha]}=\sum_{\beta\in A_{1}(k)}u_{\alpha,\beta}V_{[\beta]}.
\end{equation}
\end{definition}

The bounds for the error of the KLV method derived in Theorem
\ref{klv-error-thm} (see Lyons and Victoir \cite{lyons} for details)
assume that
the system of vector fields $V_{i}$, $i=0,\ldots, d$, satisfies the UFG
condition.
%
\begin{definition}
We define the (formal) degree of a vector field \mbox{$V_{[\alpha]},\alpha\in A$},
denoted by $d_{\alpha}$ to be the minimal integer $k$ such that
$V_{[\alpha]}$
may be written as
\[
V_{[\alpha]}=\sum_{\beta\in A_{1}(k)}u_{\alpha,\beta}V_{[\beta]}%
\]
with $u_{\alpha,\beta}\in C_{b}^{\infty}$.
\end{definition}

Note that for $\alpha\in A_{1}$ we always have $d_{\alpha}\leq\Vert
\alpha
\Vert$. It was pointed out in Crisan and Ghazali \cite{crisan} that the analysis
in Lyons and Victoir \cite{lyons} for the bound in (\ref{iterCub}) requires
$V_{0}$ to have formal degree at most 2. If the formal degree of
$V_{0}$ is
greater, the bound in (\ref{kusuokaEstimate}) changes and all bounds in the
paper will change accordingly. For sake of simplicity we will in the following
assume that $V_{0}$ has formal degree~2. The bounds can be improved in an
obvious way if the degree is $1$ or $0$. For a generalized error estimate
based on Kusuoka's ideas \cite{kusuoka3} that does not require this
additional condition, see Litterer \cite{thesis}.

A trivial generalization of Corollary 18 in Crisan and Ghazali \cite{crisan}
allows us to state a version of the Kusuoka and Stroock estimate in
terms of
the formal degree of a vector field. Let $f$ be as above and $\alpha
_{1},\ldots,\alpha_{k}\in A$ then for all $t\in(0,1]$
%
\begin{equation} \label{kusuokaEstimate}%
\bigl\Vert V_{[\alpha_{1}]}\cdots V_{[\alpha_{k}]}P_{t}f\bigr\Vert_{\infty}\leq
\frac{Ct^{1/2}}{t^{(d_{\alpha_{1}}+\cdots+d_{\alpha_{k}})/2}}\Vert\nabla
f\Vert_{\infty}.
\end{equation}
For the remainder of the paper, when we consider recombination, we are
going to
assume the following uniform H\"{o}rmander condition.\vadjust{\goodbreak}
%
\begin{definition}
We say that a collection of smooth vector fields $V_{i}$,
$i=0,\ldots,d$,
satisfies the uniform H\"{o}rmander condition (UH) if there is an
integer $p$
such that
\[
\inf\biggl\{ \sum_{\alpha\in A_{1}(p)}\bigl\langle V_{[\alpha]}(x),\xi\bigr\rangle
^{2}; x,\xi\in R^{N},|\xi|=1\biggr\} :=M>0.
\]
\end{definition}

Note that the uniform H\"{o}rmander condition implies the UFG
condition. Under
this stronger assumption it is straightforward to show that, in
addition, $P_{t}f$ is a smooth function on $R^{N}$ with explicit bounds
on its
derivatives. We outline an argument below that follows Kusuoka \cite{kusuoka2}
and gives bounds on the regularity of~$P_{t}f$, which we will use in the
following section when we apply recombination to the cubature method.

Following Kusuoka \cite{kusuoka2}, let $F(x)\in C_{b}^{\infty}(R^{N}%
;R^{N}\otimes R^{N})$ be given by
\[
F(x)=\sum_{\alpha\in A_{1}(p)}V_{[\alpha]}(x)\otimes V_{[\alpha
]}(x),\qquad
x\in R^{N},
\]
and $\lambda_{0}\dvtx R^{N}\rightarrow\lbrack0,\infty)$ be the continuous function
\[
\lambda_{0}(x)=\inf\{\langle F(x)y,y\rangle; y\in R^{N} , |y|=1\}
,\qquad
x\in R^{N}.
\]
Note that
\[
\langle F(x)y,y\rangle=\sum_{\alpha\in A_{1}(p)}\bigl\langle V_{[\alpha
]}(x),y\bigr\rangle^{2}%
\]
and hence, under the H\"{o}rmander condition (UH), we have $\lambda
_{0}(x)\geq
M>0$ for all $x\in R^{N}$. As in Kusuoka \cite{kusuoka2}, let $e_{i}%
=\{\delta_{ij}\}_{1}^{N}$ and $a_{\alpha,i}\dvtx R^{N}\rightarrow R$, $\alpha
\in
A_{1}(p),i=1,\ldots,N$, be given by
%
\begin{equation}\label{uh-coef}%
a_{\alpha,i}(x)=\bigl\langle
e_{i},F(x)^{-1}V_{[\alpha]}(x)\bigr\rangle,\qquad
x\in R^{N},
\end{equation}
and observe that
%
\begin{equation} \label{partials}%
\frac{\partial}{\partial x_{i}}=\sum_{\alpha\in A_{1}(p)}a_{\alpha
,i}V_{[\alpha]}.
\end{equation}
The following lemma may be found in Kusuoka \cite{kusuoka2}, page 274.
%
\begin{lemma}
Let $\alpha\in A_{1}(p)$, $i,i_{1},\ldots,i_{k}\in\{1,\ldots,N\}$. Then
$a_{\alpha,i}$ defined as in (\ref{uh-coef}) satisfies
%
\begin{equation}\qquad \label{uh-vf-fd}%
\biggl\vert\frac{\partial^{k}}{\partial x_{i_{1}}\cdots\partial x_{i_{k}}
}a_{\alpha,i}(x)\biggr\vert\leq CN\lambda_{0}(x)^{-(k+1)}\leq CN
\max
\bigl(M^{-(k+1)},1\bigr)
\end{equation}
for all $x$ in $R^{N}$.
\end{lemma}

The lemma shows that the functions $a_{\alpha,i}$ are in $C_{b}^{\infty}
(R^{N})$. Together with~(\ref{partials}) this immediately implies\vadjust{\goodbreak} that the
vector fields\vspace*{1pt} $\frac{\partial}{\partial x_{i}}$, $i=1,\ldots,N$, have finite
formal degree no greater than $p$. Just like identity (\ref{kusuokaEstimate}),
the following corollary is a trivial generalization of Corollary 18 in
\cite{crisan}, the result is also implicit in Kusuoka \cite{kusuoka2},
Proposition 14.
%
\begin{corollary}
\label{smoothness-uh} Suppose the vector fields $(V_{i},i=0,\ldots,d)$ satisfy
the uniform
H\"{o}rmander condition. Then for any $j\geq1$ there is a
constant $C_{2}>0$ independent of $f$ and $t$ such that
\[
\sup_{i_{1},\ldots,i_{j}\in\{ 1,\ldots,N\} }\biggl\Vert
\frac{\partial}{\partial x_{i_{1}}}\cdots\frac{\partial}{\partial
x_{i_{j}}%
}P_{t}f\biggr\Vert_{\infty}\leq C_{2}t^{-(j-1)p/2}\Vert\nabla f\Vert
_{\infty}%
\]
for all $t\in(0,1]$, $f\in C_{b}^{\infty}(R^{N})$.
\end{corollary}

We point out that the constant $C_{2}$ does (via the constant $M$ in the
H\"{o}rmander condition) depend on the underlying family of vector fields
$V_{i}$.

\section{Application to cubature on Wiener space}

\subsection{The reduction operation}\label{rec-section}

\label{main} In the iterated KLV method (Section \ref{SectionCubAlg}), the
total error $E_{\mathcal{D}}$ over the interval of approximation $[
0,T] $ is bounded by the sum of the individual errors $E_{s_{i}}$
over
smaller time intervals. The KLV method is sequential. Starting with a unit
mass particle at a single point in space time, the measures evolve
through time
by replacing each particle at time $t_{i}$ with a family of particles
at time
$t_{i+1}$. Together these new particles have the same mass as their parent
particle and are carefully positioned to provide a high order
approximation to
the diffusion of the underlying SDE. The algorithms introduced in Section
\ref{sectionalgo1} can be used very effectively to perform a global
redistribution of the mass on the particles alive at time~$t_{i}$ so
that an
essentially minimal number of particles has positive mass. At the same
time we
do not increase the one step errors $E_{s_{i}}$ significantly or affect the
order of the approximation. In this way we obtain (see Section~\ref{sectionexamples})
a~global error bound over $[ 0,T] $ for
this algorithm that is of the same order (in the number of time steps)
as the
unmodified KLV method. On the other hand, the blow up in the number of
particles is radically reduced.

The property of the intermediate measures we are targeting is to integrate
$P_{t}f$ correctly. To approximate the integral of a smooth function
such as
$P_{t}f$ with respect to a discrete measure, we need to find uniform functional
approximation schemes that apply to smooth functions on the support of this
measure. By definition, smooth functions can always be well
approximated on
balls by polynomials. However, only after one has set a fixed error bound
$\varepsilon$ and a degree for the polynomials, the size of the balls
on which
the approximation holds becomes clear. The main idea will be to
localize the
intermediate particle measures $Q$ into measures~$Q_{i}$,
where each $Q_{i}$
has its support in\vadjust{\goodbreak} such a good ball. We then replace (using the
algorithms of
Section \ref{sectionalgo1}) the measures $Q_{i}$ by reduced measures
$\tilde{Q}_{i}$ that integrate polynomial test functions of degree
at most
$r$ correctly. In that way one knows that for a smooth function $g$
\[
\sum_{i}\int g\,d\tilde{Q}_{i}%
\]
is a good approximation to $\int g\,dQ$. We subsequently prove that we can choose
the localization of the measure $Q$ in a way that ensures that we
increase the
overall bound on the error of the approximation only by a constant
factor and
examine how well we can cover the support of the intermediate measures
$Q$ by
balls for the localization.

A main idea for estimating $\varepsilon$ is to consider Taylor
expansions of
the function $P_{t}f$. We define $p$ to be the minimal integer $k$ such that
the vector fields $\{V_{\alpha},\alpha\in A_{1}(k)\}$ uniformly span $R^{N}$
at each point of $x\in R^{N}$ (as in the UH condition). For $g$ a smooth
function on $R^{N}$ let $dg\dvtx R^{N}\rightarrow\operatorname{Hom}(R^{N},R)$ denote
the full derivative of $g$. The second order derivative $d^{2}g$ is then
mapping\looseness=1
\[
R^{N}\rightarrow\operatorname{Hom}(R^{N},\operatorname{Hom}(R^{N},R))\cong
\operatorname{Hom}(R^{N}\otimes
R^{N},R).%
\]\looseness=0
The higher order derivatives can similarly be regarded as sections of
\[
\operatorname{Hom}( ( R^{N}) ^{\otimes k},R) .%
\]
We define the $r$th degree Taylor approximation of $g$ centered at
$x_{0}\in
R^{N}$ to be
%
\begin{equation} \label{taylor2}%
Tay_{r}(g,x_{0})(y)=\sum_{i=0}^{r}(d^{i}g)(x_{0})\frac{(
y-x_{0})
^{\otimes i}}{i!}
\end{equation}
and the remainder $R_{r}(g,x_{0})(y)$ by%
\[
R_{r}(g,x_{0})(y)=g( y) -Tay_{r}(g,x_{0})(y).
\]
It is clear that the $r$th degree Taylor approximation centered at
$x_{0}$ is a
polynomial of degree at most $r$. Given $u>0$ and $y\in$ $R^{N}$ let
$B(
y,u) $ denote the Euclidean ball of radius $u>0$ centered at $y.
$ Our
estimate for the remainder of the polynomial approximation is the following.
%
\begin{lemma}
\label{remainder} Let $t\in(0,1]$. The remainder function $R_{r}(P_{t}%
f,x_{0})(y)$ is uniformly bounded on $B( x_{0},u) $, that is,
\[
\bigl\Vert R_{r}(P_{t}f,x_{0})|_{B( x_{0},u) }\bigr\Vert
_{\infty}\leq C_{4}\frac{u^{r+1}}{t^{rp/2}}\Vert\nabla f\Vert_{\infty},
\]
where $C_{4}=C_{2}C_{3}$ is a constant independent of $f$, $u$ and $t$.
\end{lemma}
\begin{pf}
By Taylor's theorem we have for $y\in B( x_{0},u) $%
\[
|R_{r}(P_{t}f,x_{0})( y) |\leq\frac{\Vert d^{r+1}%
g\Vert_{\infty}}{( r+1) !}\Vert y-x_{0}
\Vert
^{r+1}\vadjust{\goodbreak}%
\]
and we note that%
\[
\Vert d^{r+1}g\Vert_{\infty}\leq C_{3}(r,N)\sup_{i_{1}%
+\cdots+i_{N}=r+1}\biggl\Vert\frac{\partial^{i_{1}}}{\partial
x_{1}^{i_{1}}%
}\cdots\frac{\partial^{i_{N}}}{\partial x_{N}^{i_{N}}}P_{t}f(y)
\biggr\Vert
_{\infty}%
\]
for some constant $C_{3}$ that only depends on $r$ and $N$. From Corollary
\ref{smoothness-uh} we see that
\[
\sup_{i_{1}+\cdots+i_{N}=r+1}\biggl\Vert\frac{\partial^{i_{1}}}{\partial
x_{1}^{i_{1}}}\cdots\frac{\partial^{i_{N}}}{\partial x_{N}^{i_{N}}}%
P_{t}f\biggr\Vert_{\infty}\leq C_{2}t^{-rp/2}\Vert\nabla f\Vert_{\infty
},%
\]
where $C_{2}$ is the constant from Corollary \ref{smoothness-uh} and
the claim follows.
\end{pf}

The bound on the remainder of the Taylor expansion of $P_{t}f$ implies that
cubature measures which integrate polynomials up to degree $r$ correctly
provide good approximations provided the support of the measure we are
targeting is contained in a sufficiently small patch.
%
\begin{proposition}
\label{uh-estimate-on-patch} Suppose the uniform H\"{o}rmander
condition is
satisfied. Let $t\in(0,1]$ and $\mu$ be a positive measure on $R^{N}$ with
finite mass $v$ satisfying $\operatorname{supp}(\mu)\subseteq B(x_{0},u)$ for some
$u>0$,
$x_{0}\in R^{N}$. Suppose a measure~$\tilde{\mu}$ is a degree $r$ cubature
measure for~$\mu$ (a reduced measure with respect to $\mu$ and the polynomials
of degree at most $r$). Then
\[
\vert E_{\mu}P_{t}f-E_{\tilde{\mu}}P_{t}f\vert\leq C_{4}%
v\frac{u^{r+1}}{t^{rp/2}}\Vert\nabla f\Vert_{\infty},
\]
where $C_{4}$ is the constant from Lemma \ref{remainder} and
independent of
$t$, $f,x_{0}$ and $u$.
\end{proposition}
\begin{pf}
We have
\begin{eqnarray*}
E_{\mu}P_{t}f-E_{\tilde{\mu}}P_{t}f
&=&
(E_{\mu}-E_{\tilde{\mu}})(Tay_{r}(P_{t}f,x_{0}))\\
&&{}+E_{\mu}R_{r}(P_{t}%
f,x_{0})-E_{\tilde{\mu}}R_{r}(P_{t}f,x_{0}).
\end{eqnarray*}
Since $\tilde{\mu}$ is a cubature measure and integrates polynomials of degree
at most $r$ correctly, the first term of the sum vanishes. Lemma
\ref{remainder} gives us the required bounds on the remaining terms.
\end{pf}

Let $\mu$ be a discrete probability measure on $R^{N}$ and $(
U_{j}) _{j=1}^{\ell}$ be a collection of balls of radius $u$ on $R^{N}
$ that covers the support of $\mu$. Then there exists a collection\vadjust{\goodbreak} of positive
measures $\mu_{j}$, $1\leq j\leq\ell$ such that $\mu_{i}\perp\mu_{j}$
for all
$i\neq j$ (i.e., the measures have disjoint support),
\[
\mu=\sum_{i=1}^{\ell}\mu_{i}\vadjust{\goodbreak}%
\]
and $\operatorname{supp}( \mu_{j}) \subseteq U_{j}\cap \operatorname{supp}( \mu)
$. We call such a collection $( U_{j},\mu_{j}) $ a localization
of $\mu$ to the cover $( U_{j}) _{j=1}^{\ell}$ and say $u$
is the
radius of the localization.
%
\begin{definition}
We say that a measure $\tilde{\mu}$ is a reduced measure with
respect to
the localization $( U_{j},\mu_{j}) _{j=1}^{\ell}$ and a finite
set of integrable test functions $P$ if there exists a localization
$(
U_{j},\tilde{\mu}_{j}) _{j=1}^{\ell}$ of $\tilde{\mu} $ such
that for $1\leq$ $j\leq\ell$ the measures $\tilde{\mu}_{j}$ are reduced
measures (see Definition \ref{def-reduced-measure}) with respect to
$\mu_{j}$
and $P$.
\end{definition}

Note that the localization of the reduced measure $\tilde{\mu}$ is with
respect to the same cover as the original measure $\mu$. It is trivial
to show
that reduced measures $\tilde{\mu}$ exist for any localization
$(
U_{j},\mu_{j}) _{j=1}^{\ell}$ of a discrete probability measure
$\mu$
and any finite set of integrable test functions $P$. Moreover, the
number of
particles in the support of $\tilde{\mu}$ is bounded above by $(
\operatorname{card}( P) +1) \ell$. The following
corollary is an
immediate consequence of Proposition \ref{uh-estimate-on-patch}. Let~$P$ in
the following be a basis for the space of polynomials on $R^{N}$ with degree
at most $r$.\looseness=-1
%
\begin{corollary}
\label{reduced-localised}Let $t<1$, $\mu$ be a discrete probability measure
on $R^{N}$ and $( U_{j},\mu_{j}) _{j=1}^{\ell}$ a
localization of
radius $u$. If $\tilde{\mu}$ is a reduced measure with respect to
$(
U_{j},\mu_{j}) _{j=1}^{\ell}$ and $P$, we have
\[
\vert E_{\mu}P_{t}f-E_{\tilde{\mu}}P_{t}f\vert\leq C_{4}%
\frac{u^{r+1}}{t^{rp/2}}\Vert\nabla f\Vert_{\infty},
\]
where $C_{4}$ is the constant from Lemma \ref{remainder} and
independent of
$t$, $f$, $u$ and the localization of radius $u$.
\end{corollary}

We define the Kusuoka--Lyons--Victoir transition (KLV) over a specified time
interval $[0,s]$, based on the cubature on Wiener space approach and already
used in the iterative method in Section \ref{SectionCubAlg}. The transition
KLV takes discrete measures on $R^{N}$ to discrete measure on $R^{N}$
and may
be interpreted as a discrete Markov kernel. Given a measure $\mu=\sum
_{i=1}^{l}\mu_{i}\delta_{x_{i}}$ on $R^{N}$ the new measure is obtained by
solving differential equations along any path in the support of the cubature
measure
\[
\sum_{i=1}^{n}\lambda_{i}\delta_{\omega_{i}}%
\]
starting from any particle in the support of $\mu$. We define
\[
\operatorname{KLV}(\mu,s)=\sum_{j=1}^{l}\sum_{i=1}^{n}\mu_{j}\lambda_{i}\delta_{\Phi
_{s,x_{j}}(\omega_{i})}.
\]

We are ready to consider recombination for the iterated KLV method. Let
$\mathcal{D}$ be a $k$ step partition $t_{0}=0<t_{1}<\cdots<t_{k}=T$ of
$[0,T]$ the global time interval of the approximation\vspace*{1pt} and recall that
$s_{j}=t_{j}-t_{j-1}$. We also let $u=(u_{2},\ldots,u_{k-1})\in
R^{k-2}$ where
each $u_{j}>0$. Let $P$ be a basis for the space of polynomials on $R^{N}$
with degree at most $r$. For each time step $s_{j}$ we first apply the KLV
method to move particles forward in time to a~measure~$Q$. We then localize
the measure $Q$ and use the algorithm of Section~\ref{sectionalgo1} to
compute a~reduced measure with respect to the localized measure and replace
$Q$ by this reduced measure. The $u_{j}$ determine the radius of the
balls in
the localization of the measure in the $j$th iteration of the method. The
polynomials in $P$ serve as the test function in the reduction.

More precisely, we define two interrelated families $Q_{\mathcal{D},u}%
^{(i)}( x) $ and $\tilde{Q}_{\mathcal{D},x}^{(i)}(
x) $ of measures. As base case we have the measures obtained by
applying twice the KLV operation starting from the point mass at $x$.
%
\begin{equation}\label{recursive-measure-def-b}
Q_{\mathcal{D},u}^{(1)}( x)
:=\operatorname{KLV}(\delta_{x},s_{1}),\qquad
Q_{\mathcal{D},u}^{(2)}( x) :=\operatorname{KLV}\bigl(Q_{\mathcal
{D},u}^{(1)}(
x) ,s_{2}\bigr) .%
\end{equation}
For\vspace*{-1pt} the recursion, the measure $\tilde{Q}_{\mathcal{D},u}^{(i)}(
x) $ is defined to be a reduced measure with respect to any fixed
localization $( U_{j},Q_{\mathcal{D},u}^{(i)}( x)
_{j}) $ of the measure $Q_{\mathcal{D},u}^{(i)}( x) $ with
radius $u_{j}$ and the set of test functions $P$ (polynomials of degree at
most $r$). We define $Q_{\mathcal{D},u}^{(i+1)}( x) $ by the
relation
%
\begin{equation}\label{recursive-measure-def}%
Q_{\mathcal{D},u}^{(i+1)}( x) :=\operatorname{KLV}\bigl(\tilde{Q}_{\mathcal{D}%
,u}^{(i)}( x) ,s_{i+1}\bigr)
\end{equation}
for all $i=2,\ldots,k-1$. Note that we do not recombine after the first and
last application of the KLV operation. The reduced measures $\tilde
{Q}_{\mathcal{D},u}^{(i)}( x) $ are not unique even after we fix
a localization of $Q_{\mathcal{D},u}^{(i)}( x) $ and a reduced
measure may be computed using the reduction algorithms of Section
\ref{sectionalgo1}.

The main result of the section is the following theorem.
%
\begin{theorem}
\label{uh-error-thm} For any choice of localizations $( U_{j}%
,Q_{\mathcal{D},u}^{(i)}( x) _{j}) $ with radius~$u_{i}$
and any reduced measures $\tilde{Q}_{\mathcal{D},u}^{(i)}( x
) $
with respect to $( U_{j},Q_{\mathcal{D},u}^{(i)}( x)
_{j}) $ and test functions $P$, $2\leq i\leq k-1$, we have
%
\begin{eqnarray}\label{uh-error-thm-eqn}
E_{\mathcal{D},k}:\!&=&\sup_{x}\bigl\vert P_{T}f(x)-E_{Q_{\mathcal{D},u}%
^{(k)}( x) }f\bigr\vert\nonumber\\
&\leq&\Biggl( C_{1}( T) \Biggl( s_{k}^{1/2}+\sum_{i=1}^{k-1}%
\sum_{j=m}^{m+1}\frac{s_{i}^{(j+1)/2}}{(T-t_{i})^{j/2}}\Biggr)\\
&&\hspace*{51.5pt}{}+C_{5}(
T) \sum_{i=2}^{k-1}\frac{u_{i}^{r+1}}{(T-t_{i})^{rp/2}}\Biggr)
\Vert\nabla f\Vert_{\infty},
\nonumber
\end{eqnarray}
where $C_{1}( T) $ and $C_{5}( T) $ are constants
independent of $f$ and the choice localizations with radius $u_{i}$.
The constant $C_{5}( T) $ can be taken equal to $C_{4}$ if
$T-t_{1}\leq1$.
\end{theorem}
\begin{pf}
The global error is bounded by
\begin{eqnarray*}
\bigl\vert P_{T}f(x)-E_{Q_{\mathcal{D},u}^{(k)}( x) }f
\bigr\vert
&\leq&\bigl\vert P_{T}f(x)-E_{Q_{\mathcal{D},u}^{(1)}( x)
}P_{T-t_{1}}f\bigr\vert\\
&&{}+\bigl\vert E_{Q_{\mathcal{D},u}^{(1)}( x) }P_{T-t_{1}%
}f-E_{Q_{\mathcal{D},u}^{(2)}( x) }P_{T-t_{2}}f\bigr\vert\\
&&{}+\sum_{j=2}^{k-1}\bigl\vert E_{Q_{\mathcal{D},u}^{(j)}( x)
}P_{T-t_{j}}f-E_{\tilde{Q}_{\mathcal{D},u}^{(j)}( x)
}P_{T-t_{j}%
}f\bigr\vert
\\
&&{}+\sum_{j=2}^{k-1}\bigl\vert E_{\tilde{Q}_{\mathcal{D},u}^{(j)}(
x) }P_{T-t_{j}}f-E_{Q_{\mathcal{D},u}^{(j+1)}( x)
}P_{T-t_{j+1}}f\bigr\vert.
\end{eqnarray*}
The first two terms and the terms in the second sum are the errors introduced
by the KLV operation and can be bounded as in the proof of Theorem
\ref{klv-error-thm}.

The terms in the first sum may each be bounded by
using Corollary \ref{reduced-localised}.
\end{pf}

The bounds for the error derived in this section assume that the
function~$f$
is Lipschitz. If $f$ has more regularity, it is clear different
estimates can
be applied to estimate the derivatives\vspace*{-1pt} of $P_{t}f$ giving alternate
bounds for~$E_{\mathcal{D},k}$. Clearly, a smaller number of balls in the
localizations of
the measures~$Q_{\mathcal{D},u}^{(j)}( x) $ reduces the
computational complexity of the method. We have not discussed yet how to
choose the localization and the degree $r$ in the reduction to optimize the
computational complexity of the method (see Section
\ref{sectionoptimisation}).

\subsection{Examples for the rate of convergence of the recombining
KLV method}\label{sectionexamples}

In this subsection we consider some particular choices of parameters
for the
recombining KLV method and examine their rate of convergence. We first
fix for
the remainder of this section (a family of) partitions $\mathcal{D}$
for the
time interval $[0,T]$. We recall a family of uneven partitions from
Lyons and
Victoir~\cite{lyons} which has smaller time steps toward the end and is given
by
%
\begin{equation}\label{fixedpartitions}%
t_{j}=T\biggl( 1-\biggl( 1-\frac{j}{k}\biggr) ^{\gamma}\biggr) .
\end{equation}
For $\gamma>m-1$ the results in \cite{lyons} (see also Kusuoka \cite
{kusuoka3}%
) show that
%
\begin{equation}\label{rateofconvergence}%
s_{k}^{1/2}+\sum_{i=1}^{k-1}\sum_{j=m}^{m+1}\frac{s_{i}^{(j+1)/2}}%
{(T-t_{i})^{j/2}}\leq C_{6}( m,\gamma) T^{1/2}k^{-(m-1)/2},
\end{equation}
while for the case $0<\gamma<m-1$ one obtains
\[
s_{k}^{1/2}+\sum_{i=1}^{k-1}\sum_{j=m}^{m+1}\frac{s_{i}^{(j+1)/2}}%
{(T-t_{i})^{j/2}}\leq C_{7}( m,\gamma) T^{1/2}k^{-\gamma/2}.
\]
In the following two examples we work with the partition defined in
(\ref{fixedpartitions}) and the notation of Theorem \ref{uh-error-thm}. Using
this particular choice of partitions ensures that the bound on the KLV error
is of high order in the number of iterations $k$.
%
\begin{example}
\label{example-convergence}Let $\gamma>m-1$, $r=$ $\lceil m/p\rceil$
and\vspace*{-1pt}
$u_{j}=s_{j}^{p/2-a}$, where $a:=\frac{p-1}{2(\lceil m/p\rceil+1)}\geq
0$. Then%
%
\begin{eqnarray}
&&\sup_{x}\bigl\vert P_{T}f(x)-E_{Q_{\mathcal{D},u}^{(k)}( x)
}f\bigr\vert\nonumber\\
&&\qquad\leq\Biggl( C_{1}( T) \Biggl( s_{k}^{1/2}+\sum_{i=1}^{k-1}%
\sum_{j=m}^{m+1}\frac{s_{i}^{(j+1)/2}}{(T-t_{i})^{j/2}}\Biggr)
\nonumber\\[-8pt]\\[-8pt]
&&\qquad\quad\hspace*{31.4pt}{}+C_{5}(
T) \sum_{i=2}^{k-1}\frac{s_{i}^{(\lceil m/p\rceil p+1)/2}}%
{(T-t_{i})^{\lceil m/p\rceil p/2}}\Biggr) \Vert\nabla f\Vert\nonumber\\
&&\qquad\leq C_{8}k^{-(m-1)/2}T^{1/2}\Vert\nabla f\Vert_{\infty},%
\nonumber
\end{eqnarray}
where $C_{8}=C_{6}( m,\gamma) ( C_{1}( T)
+C_{5}( T) )$.
\end{example}

Note that $0$ $\leq p/2-a\leq p/2$ for all positive integers $p$ and
$m$ and
that for $s_{j}\leq1$ we have $u_{j}\geq s_{j}^{p/2}$. In the
next\vspace*{2pt}
example we
choose the radius of the balls in the reduction operation such that at each
step in the iteration the bound on the recombination error matches the bound
on the KLV error.
%
\begin{example}
Let $\gamma>m-1,m=r$, that is, the degree of the polynomials used in
the reduction
operation equals the degree of the cubature in the KLV method. Let
$u_{j}$,
$j=2,\ldots, k-1$ be given by%
\[
u_{j}=\biggl( \frac{s_{j}^{m+1}}{( T-t_{j}) ^{m-rp}}\biggr)
^{{1}/({2( r+1) })}.
\]
Then
%
\begin{eqnarray}
&&
\sup_{x}\bigl\vert P_{T}f(x)-E_{Q_{\mathcal{D},u}^{(k)}( x)
}f\bigr\vert\nonumber\\
&&\qquad\leq\Biggl( C_{1}( T) \Biggl( s_{k}^{1/2}+\sum_{i=1}^{k-1}%
\sum_{j=m}^{m+1}\frac{s_{i}^{(j+1)/2}}{(T-t_{i})^{j/2}}\Biggr)
\nonumber\\[-8pt]\\[-8pt]
&&\qquad\quad\hspace*{54pt}{}+C_{5}(
T) \sum_{i=2}^{k-1}\frac{s_{i}^{(m+1)/2}}{(T-t_{i})^{m/2}}\Biggr)
\Vert\nabla f\Vert\nonumber\\
&&\qquad\leq C_{9}k^{-(m-1)/2}T^{1/2}\Vert\nabla f\Vert_{\infty},
\nonumber
\end{eqnarray}
where $C_{9}=C_{6}( m,\gamma) ( C_{1}( T)
+C_{5}( T) )$.\vadjust{\goodbreak}
\end{example}

As before, if $T-t_{1}<1$, the constants $C_{8}$ and $C_{9}$ can be
taken to be
$C_{6}( m,\gamma) ( C_{1}( 1) +C_{4})
$. The parameters chosen in the above examples guarantee high order
convergence, but are not necessarily computationally optimal. In the following
section we examine how, for a fixed error $\varepsilon$, the choice of
$r$ and
$u$ can be varied to be closer to the optimal computational effort in the
recombination operation.

\subsection{An optimization}\label{sectionoptimisation}

This paper establishes stable higher order particle approximation methods
where the computational effort involved grows polynomially with the
number of
time steps when the number of steps is large and the underlying system remains
compact (see Section \ref{sect-growth}). In concrete examples, an optimization
of the different aspects of this algorithm, under the constraint of fixed
total error, leads to even more effective approaches; although we
expect that
different problems would benefit from different distributions of the
computational effort. For example, there is a~trade-off between the
degree of
the polynomials that are used as test functions and the size of the
balls used
to define the localization of the measure for the recombination (smaller
patches if we use higher degree polynomials in the test functions and
we fix
the error of the approximation).

Specifically, suppose we are given a discrete measure $\mu$ and the property
we care about is the integral of $\mu$ against a smooth function $g $.
As in
our application to the KLV method we consider a reduced measure
$\tilde{\mu}$ (Definition~\ref{def-reduced-measure}) with
respect
to the
polynomials of degree at most $r$ and a localization of $\mu$ with
radius at
most~$\delta$. The number of balls of radius $\delta$ required to cover the
support of $\mu$ is at most of order $( \frac{D}{\delta}) ^{N}$,
where $D$ is the diameter of $\operatorname{supp}( \mu)$. Let $\varepsilon
$ be
the error of the approximation of $\int g\,d\mu$ by $\int
g\,d\tilde{\mu}$.

Note that
\[
\varepsilon=\frac{\delta^{r+1}c_{r+1}}{( r+1) !}
\]
for some $c_{r+1}\leq{\sum_{i_{1}+\cdots+i_{N}=r+1}}\Vert\frac
{\partial^{i_{1}}}{\partial x_{1}^{i_{1}}}\cdots\frac{\partial^{i_{N}}%
}{\partial x_{N}^{i_{N}}}g\Vert_{\infty}$. Fixing the error
$\varepsilon$ gives a~simple relation for $\delta$ and $r$%
%
\begin{equation}\label{constr-2}%
\delta=\biggl( \frac{\varepsilon( r+1) !}{c_{r+1}}\biggr)
^{1/(r+1)}.
\end{equation}
Let $\hat{n}$ be the number of particles in the support of $\mu$. The
computational complexity of the recombination operation as a function of
$\delta$, $\hat{n}$ and $r$ is at most of order
\[
\biggl( \frac{D}{\delta}\biggr) ^{N}\pmatrix{r+N\cr N}^{4}\log\hat
{n}+\hat
{n}\pmatrix{r+N\cr N}
\]
which may be optimized subject to the constraint (\ref{constr-2}).

Note that in our application to cubature on Wiener, $\mu$
corresponds\break
to~$Q_{\mathcal{D},u}^{(j)}( x) $ and the function $g$ is given by
$P_{T-t_{j}}f$. The calculation above also allows us to decide after
each step
of the iteration if it is of computational benefit to carry out a (full)
recombination operation.

\subsection{Simple bounds on the number of test functions; covering the
support of the particle measures}\label{sect-growth}

In this section we obtain upper bounds for the number of ODEs required to
solve in the recombining KLV method with $k$ iterations. For this, it is
sufficient to bound the number of balls in the cover of the
localizations of
the particle measures uniformly for all $k$ iterations. We first find a large
ball $B(x,\rho)$ that covers $\operatorname{supp}(Q_{\mathcal{D},u}^{(j)}( x)
)$, $j=1,\ldots,k-1$, and then estimate the number of balls that are required
to cover $B(x,\rho)$. The balls in the covers of the localizations will have
to be sufficiently small to preserve the high order accuracy of the
method. We
can show that under the assumption that the vector fields $V_{i}$ are bounded
and satisfy the UH condition, we have a high order method and the computational
complexity is polynomial in $k$ the number of iterations. Similar
results can
be obtained if the underlying system remains compact.

The following theorem demonstrates that we can achieve the same rate of
convergence in the number of iterations $k$ as in Kusuoka's algorithm
and the
vanilla KLV method, but control the complexity of the method by an explicit
polynomial in $k$. This compares to exponential growth in the vanilla
KLV
method without recombination, which despite its exponential growth
leads to
numerically highly effective algorithms (see, e.g., Ninomiya and
Victoir~\cite{ninomiya-victoir}). The estimates in this section are not
designed to be
optimal and can be improved. Closer to optimal choices for the radius $u_{i}$
and degree~$r$ in the reduction operation have been discussed in Section
\ref{sectionoptimisation} and may be used to decide if it is computationally
efficient to recombine the particle measure at time $t_{i}$.
%
\begin{theorem}
\label{poly-growth}Suppose the uniform H\"{o}rmander condition is satisfied
and the vector fields $V_{i}$ are uniformly bounded by some constant
$M^{\prime}>0$. We can achieve
%
\begin{equation}\label{rate-cover}%
E_{\mathcal{D},k}=\sup_{x\in R^{N}}\bigl\vert P_{T}f(x)-E_{Q_{\mathcal
{D}%
,u}^{(k)}( x) }f\bigr\vert\leq C_{8}k^{-(m-1)/2}T^{1/2}%
\Vert\nabla f\Vert_{\infty},
\end{equation}
  while the number of test functions in the reduction operation, and
hence the number of elementary ODEs to solve grows polynomially in $k$.
\end{theorem}
\begin{pf}
Let $m>0$ be the degree of the cubature in the KLV method. Fix the partition
$\mathcal{D}$ to (\ref{fixedpartitions}) for some
$\gamma>m-1$. As in Example \ref{example-convergence}, let $r=\lceil
m/p\rceil$
and $u_{j}=s_{j}^{p/2-a}$, $a=\frac{p-1}{2(\lceil m/p\rceil+1)}\geq0$
in the
reduction\vspace*{1pt} operation. We note that the error $E_{\mathcal{D},k}$ satisfies
(\ref{rate-cover}) and it remains to show that the
number of
particles in support of the measures $Q_{\mathcal{D},u}^{(k)}(
x) $ grows polynomially in $k$, which is equivalent to the number of
balls in the localizations growing polynomially in $k$.

Note that if $\omega\in C_{0}^{0}([0,1],R^{d})$ is a continuous path of
bounded variation of length $L$, we have
\[
\vert x-\Phi_{1,x}(\omega)\vert\leq M^{\prime}L,
\]
where $\Phi$ is the It\^o functional defined in (\ref
{ito-functional}), that is, $\Phi_{1,x}(\omega)$ is the point we obtain by
solving the
equation (\ref{sde}) along the path $\omega$ starting at $x$.
Let $L$ be given by $L=\max_{i=1,\ldots,n}$length$( \omega
_{i})
$, the maximum of the lengths of the paths in the support of the degree $m$
cubature formula on Wiener space over the unit time interval. Observe
that by
construction any particle in the support of $Q_{\mathcal
{D},u}^{(j)}(
x) $ [compare the definition of the measures in
(\ref{recursive-measure-def})] may be written~as%
\[
\Phi_{\sum_{i=1}^{j}s_{i},x}(\omega_{s_{1},i_{1}}\otimes\cdots\otimes
\omega_{s_{j},i_{j}})
\]
some $i_{1},\ldots,i_{j}\in\{ 1,\ldots,n\} $, the $\omega_{s,i}$
are the rescaled paths defined in (\ref{rescaled-paths})
and~$\otimes$ denotes to the concatenation of paths. For $k$ sufficiently
large we
may assume $s_{i}<1$ and we deduce that
\[
\operatorname{supp}\bigl(Q_{\mathcal{D},u}^{(j)}( x) \bigr)\subseteq
B\Biggl(
x,M^{\prime}L\sum_{i=1}^{j}s_{i}^{1/2}\Biggr) \subseteq B(
x,M^{\prime
}LkT^{1/2}) .
\]
In the reduction operations we consider a basis of the polynomials of degree
at most $\lceil m/p\rceil$ and the measure is localized by balls of radius
$u_{j}$ which need to cover $\operatorname{supp}(Q_{\mathcal{D},u}^{(j)}( x
) )$.\vspace*{-1pt}
For $s_{j}<1$, that is, for $k$ sufficiently large, we have $u_{j}\geq
s_{j}%
^{p/2}$ and for our uneven family of partitions $\min_{j=2,\ldots, k-1}%
s_{j}^{p/2}<s_{k}^{p/2}=( T/k^{\gamma}) ^{p/2}$. Thus, the number
of particles in each\vspace*{1pt} of the reduced measures is uniformly bounded above by
${\lceil m/p\rceil+N\choose N}$ times the number of balls of radius
$(
T/k^{\gamma}) ^{p/2}$ required to cover the ball $B(
x,M^{\prime
}LkT^{1/2}) $ in $N$-dimensional space, which is a polynomial of degree
at most $N( \gamma p/2+1) $ in~$k$.
\end{pf}

Similarly, we can derive a result analogous to Theorem
\ref{poly-growth} if the
underlying system remains compact.

\begin{appendix}\label{app}
\section*{Appendix: A numerical toy example}

We consider a linear one-dimensional problem. The boundary data is Lipschitz,
piecewise smooth, and the locations of the discontinuities in the derivatives
are not known to the program. The answer is required to high accuracy.
In our
test case we applied the approximation method to the heat equation with
boundary data%
\[
f( x) =\max( 1-e^{x},0) ,\vadjust{\goodbreak}
\]
which corresponds to the calculation of a Black--Scholes put option at
logarithmic scale. We considered a time horizon of $T=1$ and various initial
conditions $X_{0}\in[ -4,4]$. We set our goal to achieve an
accuracy of $10^{-10}$. This example is particularly suitable as a test example
because the solution to the equation is known in closed form in terms
of well
known special functions which can be used to determine the precise
error in
the approximation.

We applied a modified form of the KLV method with recombination
introduced in
this paper. For $\theta<1$ consider a geometrically converging
partition of
the unit time interval given by
\[
1-t_{j}=( 1-\theta) ( 1-t_{j-1}) ,\qquad
j=1,\ldots,k-1,
\]
$t_{0}=0$ and $t_{k}=1$. Note that the length of the time steps $s_{j}$
in the
partition is given by $s_{j}=\theta( 1-t_{j-1})$. In our
particular example we chose $\theta$ to be $0.4$. To achieve the required
accuracy we used a $15$ point Gaussian quadrature which we had previously
computed to high accuracy. For the heat equation, the particles of the cubature
approximation are given by the Gaussian quadrature and we do not
require to
solve ODEs. As described in Section~\ref{rec-section} we used
polynomial test
functions of degree $m$ and localized the support of intermediate particle
measures in the approximation. We then used a heuristic based on the
information provided by the $W^{1,1}$ norm of $f$ to determine, as
outlined in
Section \ref{sectionoptimisation}, the degree of polynomial approximation
that minimizes the computational complexity of the overall reduction process
subject to achieving the required accuracy.

In addition, we modified the algorithm to make use of the piecewise smooth
nature of the boundary data. The algorithm compares for each particle a two
step KLV with a one step KLV estimate to the boundary. If both approximates
agree to the error tolerance, the algorithm immediately leaps to the boundary.
As the required accuracy is close to machine precision, false positives are
very unlikely. Recombination is then performed on the remaining particles.

\begin{table}
\caption{Absolute error and computational effort for the approximation
of $u(x,1)$ for different values of $x$}
\label{table1}
\begin{tabular*}{\tablewidth}{@{\extracolsep{\fill}}lcccc@{}}
\hline
$\bolds{X_{0}}$ & \multicolumn{1}{c}{$\bolds{-4}$} &
\multicolumn{1}{c}{$\bolds{-3}$} & \multicolumn{1}{c}{$\bolds{-2}$}
& \multicolumn{1}{c@{}}{$\bolds{-1}$}\\
\hline
Absolute error $\varepsilon$ & 3.186E--11 & 1.01E--11 & 4.962E--11 &
1.2014E--10\\
Evaluations at the boundary & 1,410,075 & 1,416,600 & 1,426,050
& 1,432,350\\
Particles & 94,005 & 94,440 & 95,070 & 95,490
\end{tabular*}

\begin{tabular*}{\tablewidth}{@{\extracolsep{\fill}}lccccc@{}}
\hline
$\bolds{X_{0}}$ & \multicolumn{1}{c}{$\bolds{0}$}
& \multicolumn{1}{c}{$\bolds{1}$} & \multicolumn{1}{c}{$\bolds{2}$}
& \multicolumn{1}{c}{$\bolds{3}$} & \multicolumn{1}{c@{}}{$\bolds{4}$}\\
\hline
$\varepsilon$ & 3.612E--11 & 4.173E--12 & 2.52E--11 & 5.47E--11 & 5.62E--12\\
Evaluations & 1,430,775 & 1,425,600 & 1,424,700 & 1,417,725 & 1,418,175\\
Particles & 95,385 & 95,040 & 94,980 & 94,515 & 94,545\\
\hline
\end{tabular*}
\end{table}
%

In order to achieve an accuracy of $10^{-10}$ we chose $m=8$ and a
radius for
the localization that was proportional to $\sqrt{1-t}$ and covered the
surviving measure with approximately 13 nonempty\vspace*{1pt} components in the
localization. The runtime of our single threaded C$++$ code\footnote{As measured
on a Lenovo Thinkpad x201t notebook computer. We used intel mkl for the lapack
support and this might use omp internally.} was between 0.5 and 0.6~s. The
parameter restricting the maximal depth of the approximation tree was
set to
$28$. Table \ref{table1} and Figure~\ref{figure1} summarize the absolute
error of
the approximation, the number of reduced particles inside the domain
and the
total number of evaluations of the cubature at the boundary for various values
of $X_{0}$. Note that the number of particles compares to $\sim$15$^{27}$
internal particles for the vanilla cubature algorithm and even if combined
with a~partial sampling scheme such as the tree based branching
algorithm one
could not hope to compute an approximation to ten digit accuracy.%

%
\begin{figure}[b]
\includegraphics{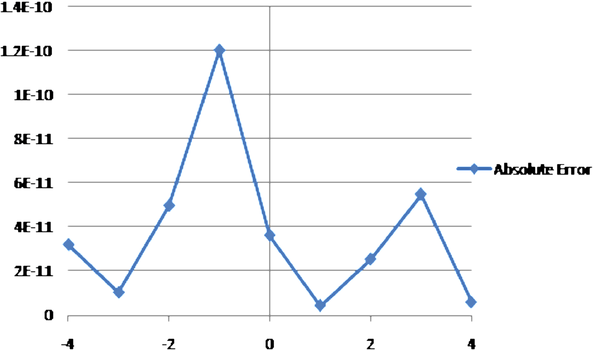}
\caption{Absolute error for the approximation of $u(x,1)$ for different values of $x$.}%
\label{figure1}%

\end{figure}
%

Even though the problem we have considered is merely a toy example, computing
the solution to high accuracy with a vanilla off the shelf PDE solver appears
to be nontrivial. However, a fair comparison must involve at least adaptive
methods; we were afraid to do this ourselves as it would not carry much weight
because we do not have the computational expertise to get good outcomes from
these packages. So we were very grateful that our colleague Kathryn
Gillow in
Oxford was willing to give it a quick spin on adaptive software she had
developed with Endre Suli.

She says: ``I've now tried a few approaches to solving your problem but can't
get results even close to yours in terms of accuracy achieved in such a small
amount of CPU time. In all cases I've solved the heat equation on the spatial
interval $-9.9<x<10.1$ (so that with a coarse uniform mesh the point
$x=0$ was
not a node). Then to look at the error I have computed the solution at time
$1$ and for $x$ integer between $-5$ and $5$ as you suggest. The~first
approach I took was to do an adaptive finite element solution with the
adaptivity geared toward getting an accurate solution at time $t=1$.
The mesh
can change at every timestep which is obviously less than ideal as you then
need to keep recomputing the matrices. The code is taking about $30$ seconds
and giving accuracy of between $10^{-4}$ and $10^{-7}$ depending on which
integer you look at. It actually turns out to be more efficient to do
something a bit more naive, namely, to adapt the mesh to resolve the initial
condition well and then use that mesh for the rest of the computation. As
expected, this clusters the nodes around $x=0$ and the mesh is fairly coarse
elsewhere. The advantage of this is that you just solve the same matrix
problem at every time-step. This speeds things up a lot without
degrading the
accuracy for this problem. So here I'm getting accuracy of between $10^{-4}$
and $10^{-6}$ in about $1$ second. Then, finally, I gave Nick Trefethen et
al's Matlab package Chebfun a go. In order to solve the heat equation which
exploits the fact that the problem is linear so you can write the
solution at
a given time $t$ as $\operatorname{exp}(t\ast L)u_{0}$ where $L$ is the spatial operator
(including boundary conditions) and $u_{0}$ is the initial condition.
It seems
that Chebfun struggles when $u_{0}$ is not smooth and it actually turns
out to
be more efficient to compute the solution at time $t=1$ in two stages, namely,
$u(x,dt)=\operatorname{exp}(dt\ast L)u_{0}$, $u(x,1)=\operatorname{exp}((1-dt)\ast L)\ast u(x,dt)$.
The best
accuracy using this approach is $5\ast10^{-6}$ taking $6.5$ seconds. Chebfun
does a lot better when you have smooth initial data. Then it can solve the
same type of problem in $0.1$ s giving errors of $10^{-7}$.''

No doubt the approach we take tries to do less than that taken by our
colleagues, (it only computes the solution at the required points,
etc.) and we
have tried to polish the code for our problem but still we find it encouraging
evidence that this paper is putting ideas together in a novel way. The linear
algebra we do is numerically really heavy, but it seems to pay.
\end{appendix}


%

\printaddresses

\end{document}